\documentclass[10pt]{amsart}
\usepackage{geometry,enumerate}                % See geometry.pdf to learn the layout options. There are lots.
\usepackage{graphicx}
\usepackage{amssymb}
\usepackage{epstopdf}
\usepackage{hyperref}
\usepackage{amsrefs}

\theoremstyle{plain}% default
	\newtheorem{thm}{Theorem}[section]
	\newtheorem{lem}[thm]{Lemma}
	\newtheorem{prop}[thm]{Proposition}
	\newtheorem*{cor}{Corollary}

	\theoremstyle{definition}
	\newtheorem{defn}{Definition}[section]

	\theoremstyle{remark}
	\newtheorem*{rem}{Remark}
	\newtheorem{remnmb}{Remark}

	\newtheorem*{notation}{Notation}

\title[Finite Distributive Groups]{\large{Distributive Gruppen endlicher Ordnung}\\\vspace{2cm}\normalsize{Finite Distributive Groups}}
\author{Burstin, C. \& Mayer, W.}
\date{Received November 1927}                                           % Activate to display a given date or no date
\translator{Ansgar Wenzel, Independent Researcher, ansgar@wenzel.direct}
\begin{document}

\begin{abstract}
This is a translation of \cite{original}. I have added translations for (possibly) outdated definitions in an appendix at the end.

In this paper, we define distributive groups and show some properties of them. We then concern ourselves with the homogeneity of distributive groups, before showing how they can be generated from both associative and distributive groups. After that, we look at subgroups and define an index theorem for distributive groups before describing the structure of distributive groups. Finally, we present two addenda with several theorems that we proved while trying to prove that $\left|(A.p).(A.q)\right|=\left|A\right|$.
\end{abstract}
\maketitle
\section{The Axiomatic of distributive groups}\label{chapter1}
Let there be a system of finitely or infinitely many elements:
\begin{equation*}
a\;b\;c\;\ldots
\end{equation*}
 and a composition that for two elements $a, b$ in a certain order $a, b$ assigns the element $a.b$, the result of the composition of $a$ and $b$. If the following three Axioms hold, we call such a system a \textit{distributive group}.
\begin{enumerate}[I. {Axiom}]
 \item\label{Axiom1} The composition result $a.b$ of any two elements $a$ and $b$ of the system is itself an element of the system.
 \item\label{Axiom2} If $a$ and $b$ are any two elements of the system, then the two equations 
 \begin{equation*}
 a.x = b \text{ and } y.a=b
 \end{equation*}
 have exactly one and only one solution in the system.
 
 (Thus, this Axiom requires the existence and uniqueness of the inverse operations).
 \item\label{Axiom3} Let $a$, $b$ and $c$ be any three elements of the system then there exist the following two, identical relations
 \begin{equation*}
 (a.b).c = (a.c).(b.c) \text{ and } c.(a.b) = (c.a).(c.b)
 \end{equation*}
 \end{enumerate}
 Axioms \ref{Axiom1} and \ref{Axiom2} are Axioms of classical group theory; instead of Axiom \ref{Axiom3} we have the Axiom of associative composition: $(a.b).c = a.(b.c)$. Thus, in order to distinguish them from distributive groups, we will call the classical groups \textit{associative} groups.
 
If there exists such an Axiomatic system, the questions of independence and (non-) contradiction of the Axioms naturally arise.
We show the independence of Axiom \ref{Axiom1} from Axioms \ref{Axiom2} and \ref{Axiom3}:

Let the system of elements be the (positive and negative) integers; let the composition be $a.b = \frac{a+b}{2}$. Here, Axioms \ref{Axiom2} and \ref{Axiom3} hold but Axiom \ref{Axiom1} does not.

In a system of \textit{finitely many} elements, Axiom \ref{Axiom1} follows from \ref{Axiom2}.

A second example shows the independence of Axiom \ref{Axiom2} from Axioms \ref{Axiom1} and \ref{Axiom3}.

Let the system consist of $n$ elements, $a_1,a_2,\ldots,a_n$ and let the composition be $a_i.a_k = a_k$. Here, Axioms \ref{Axiom1} and \ref{Axiom3} hold but Axiom \ref{Axiom2} does not.

Every associative group with more than one element is an example for the independence of Axiom \ref{Axiom3} from the other two. For this, we will show that such an associative group cannot be distributive.

The non-contradiction of the Axioms is shown by the creation of distributive groups in the following examples:
\begin{enumerate}
\item Let the system of elements be the complex numbers. To the numbers $a$ and $b$, the composition $a.b$ shall assign  the number
\begin{equation*}
a.b = \alpha a + \beta b
\end{equation*}
where $\alpha$ and $\beta$ are fixed and such that $\alpha + \beta = 1$.
\item Let the system of elements be the positive real numbers (except zero) and let the composition $a.b$ assign to the numbers $a$ and $b$ the number $a.b=\sqrt{ab}$ (the geometric mean). More generally, $a.b = a^\alpha b^\beta$ with $\alpha + \beta = 1$, $\alpha,\beta\neq0$. Here, $a^\alpha = e^{\alpha\ln\alpha}$. 
\item Let the system of elements be the points of the $n$-dimensional affine space. Let the composition $a.b$ assign to the points $a$ and $b$ a point on the line $\overline{ab}$ that divides this line by a certain proportion $\alpha:\beta$ (e.g. the mean of the line $\overline{ab}$).
\item Let the system of elements be the points of the $n$-dimensional projective space except the points on an $(n-1)$ dimensional hyperplane $E_{n-1}$ of this space\footnote{\textit{Dual}: Let the system of the elements be the $E_{n-1}$ of a projective $R_n$ except the $E_{n-1}$ going through a certain point in $R_n$.}.

Then the point $a.b$ shall be on the line $ab$ such that the cross-ratio $(a,b,a.c,b.c)$, where $c$ is the intersection point between the line $ab$ and $E_{n-1}$, has a certain, fixed, value $\chi$.

If one introduces projective coordinates, if $a_1,\ldots,a_{n+1};b_1,\ldots,b_{n+1}$ are the coordinates of the points $a$ and $b$ and if $A_ix_i=0$, $i=1,\ldots,n+1$ is the equation of $E_{n-1}$, then the point $a.b=d$ has the projective coordinates
\begin{equation*}
d_i=A_i(a_ib_t - \chi b_i a_t) \text{ where } i,t = 1,\ldots,n+1
\end{equation*}
Showing that this is indeed a group follows easily from this formula; we will show a very simple method that can be used for this in section \ref{chapter2}.
\end{enumerate}

Before starting to give examples of finite distributive groups, i.e. groups with finitely many elements, we will discuss a characteristic property of distributive groups. In an associative group, there is always one element, the unit element with notation $e$ for which it holds that $a.e=e.a=a$ for each element $a$ in the associative group. In a distributive group there is no such element, rather, there is a such \textit{homogeneity} that any property of one element of this group holds for each element of this group (see section \ref{chapter2}).

We want to show that any non-trivial distributive group (i.e. with more than one element) cannot include a unit element. If we set in one of the relations of Axiom \ref{Axiom3} $b=c=a$, with $a$ any element in the group, then we have $(a.a).a = (a.a).(a.a)$ and hence, because of Axiom \ref{Axiom2} the important relation
\begin{equation*}
a.a=a
\end{equation*}
which thus holds for all elements in a distributive group. If now a unit element $e$ were to be an element of this group, then we would have $a.e=a.a$ and hence $a=e$, that is, every element in this group would be the unit element and hence the group would be trivial. \textit{q.e.d.}

The trivial group with only one element $a$ for which Axiom \ref{Axiom1} holds via $a.a=a$ is hence an example of a group which is both associative and distributive.

(Such a group can only have one element since otherwise it would not be distributive. On the other hand, in order to be associative it must include the unit element.)

\subsubsection*{A distributive group with only two elements does not exist.} 

Let $a$ and $b$ be those two elements, then from $a.a=a$ and $b.b=b$ we get that $a.b\neq a,b$ so that Axiom \ref{Axiom1} does not hold.

\subsubsection*{The distributive group with three elements exists.} In fact, it is commutative.
Let $a,b,c$ be the elements of this group, so we have that $a.a=a$, $b.b=b$ and $c.c=c$. $a.b\neq a,b$ so $a.b=c$. Similarly, $a.c =b$ and $b.c=a$. The Cayley table of this group is
\begin{center}
\begin{tabular}{c|ccc}
&a&b&c\\
\hline
a&a&c&b\\
b&c&b&a\\
c&b&a&c
\end{tabular}
\end{center}
\subsubsection*{Now we show that every finite commutative distributive group has odd order.}
Let $G\left\{a_1,a_2,\ldots,a_n\right\}$ be a commutative distributive group of order $n$. We extract one element, $a_1$, say and have a closer look at the remaining elements, $a_2,\ldots,a_n$.

Let $a_i$ be an element of this remaining elements, then there exists an "assigned" element $a_{\sigma_i}$ in $G$ such that
\begin{equation*}
a_i.a_{\sigma_i} = a_1
\end{equation*}
Since $a_i\neq a_1$, $a_{\sigma_i}\neq a_1,a_i$. Thus, $a_{\sigma_i}$ is an element of the remaining elements which is not equal to $a_i$. Since the group is commutative, $a_{\sigma_i}$ has $a_i$ as its assigned element. Thus, the elements of the remaining elements are paired in assigned element pairs and hence the order of the remaining elements is even. Therefore, the order of $G$ is odd, \textit{q.e.d.}
\subsubsection*{Conversely, for each odd number $2N+1$ there exists a commutative distributive group of this order.}
We call the elements of this group $1,2,\ldots,2N+1$ and let the composition element $a.b$ of the elements $a$ and $b$ be as follows
\begin{equation*}
a.b\equiv (n+1) (a+b) \mod (2n+1)
\end{equation*}
It is easy to see that this is indeed a commutative group.

Geometrically, this is the group of vertices of a $2n+1$-gon where the composition element $c=a.b$ is then always lying on the perpendicular bisection of the line with endpoints $a$, $b$.
\begin{proof}
Since $a(n+1)\equiv \frac{a}{2}\mod (2n+1)$, we can rewrite $a.b$ as
\begin{equation*}
a.b\equiv \frac{a+b}{2}\mod (2n+1)
\end{equation*}
Now let $a<b$. Then
\begin{enumerate}[1.]
\item \begin{equation*}b=a+2d,\ c\equiv \frac{2a+2d}{2}=a+d\mod(2n+1)\end{equation*} that is,\begin{equation*} c=a+d, b=c+d\end{equation*}
\item \begin{equation*}b=a+2d+1,\ c\equiv\frac{a+b+2n+1}{2}\equiv a+d+n+1\equiv b+(n-d)\mod (2n+1)\end{equation*} that is, \begin{equation*}c=b+(n-d), c+(n-d)=b+2(n-d)=a+2n+1\equiv a\end{equation*}
\end{enumerate}
\end{proof}

Even though there are no commutative groups of even order, there are non-commutative ones. The \textit{group of order four} is one example. If we let $a_1,a_2,a_3,a_4$ be the elements of this group, we can always assume $a_1.a_2=a_3$ (since $a_1.a_2$ can only be $a_3$ and $a_4$ and we can always renumber to get $a_1.a_2=a_3$). This means that in the Cayley table the following compositions are known:
\begin{center}
\begin{tabular}{c|cccc}
&$a_1$&$a_2$&$a_3$&$a_4$\\
\hline
$a_1$&$a_1$&$a_3$&.&.\\
$a_2$&&$a_2$&&\\
$a_3$&&&$a_3$&\\
$a_4$&&&&$a_4$
\end{tabular}
\end{center}
In the first row, $a_2$ and $a_4$ have to be in the two free places. Since $a_1.a_4\neq a_4$, $a_1.a_3=a_4$ and $a_1.a_4=a_2$. With the help of Axiom \ref{Axiom2}, one can fill in the other free places as well and this leads to:
\begin{center}
\begin{tabular}{c|cccc}
&$a_1$&$a_2$&$a_3$&$a_4$\\
\hline
$a_1$&$a_1$&$a_3$&$a_4$&$a_2$\\
$a_2$&$a_4$&$a_2$&$a_1$&$a_3$\\
$a_3$&$a_2$&$a_4$&$a_3$&$a_1$\\
$a_4$&$a_3$&$a_1$&$a_2$&$a_4$
\end{tabular}
\end{center}
With this table, it is easy to show that Axiom \ref{Axiom3} holds. Thus, there is one and only one distributive group of order four.

There is, again, only one group of order five, the commutative one described earlier, while there is no group of order six.

The question of classifying all groups of a given order seems to be as hard for distributive groups as is the equivalent question for associative groups. There will be more examples in section \ref{chapter3}.

\section{homogeneity of distributive groups}\label{chapter2}
Let $G$ be a distributive group and let $A=\left\{a_1,a_2,\ldots\right\}$ be a subgroup of $G$ ($A$ does not have to be countable. The notation $a_i$ is chosen for simplicity). Let $p$ be any element of $G$, then $B=A.p = \left\{b_1=a_1.p,b_2=a_2.p,\ldots\right\}$ is a subgroup of $G$, isomorphic to $A$.

Here, similarly to classical group theory, we call two groups $\left\{a_1,a_2,\ldots\right\}$ and $\left\{b_1,b_2,\ldots\right\}$ uniquely isomorphic\footnote{In the original, \textit{einstufig isomorph}. For a definition, see Definition \ref{defn:isomorph}} if a bijective map $a_i\leftrightarrows b_i$ can be defined such that to the composition $a_i.a_k$ of $a_i$ and $a_k$ can be assigned a unique composition $b_i.b_k$ of $b_i$ and $b_k$, where $b_i$ is assigned to $a_i$ and $b_k$ is assigned to $a_k$.
\begin{proof}
If $A$ is a finite group, $a_i.a_k=a_l$ implies via a right-sided composition with $p$: $b_i.b_k=b_l$. The Cayley table of the system $\left\{b_1,b_2,\ldots\right\}$ is therefore the same as the Cayley table of the system $\left\{a_1,a_2,\ldots\right\}$ if the letter $b$ is replaced with the letter $a$. Thus, $B$ is a subgroup of $G$, isomorphic to $A$.

For infinite groups, the same proof holds via an extension of Cayley tables to infinite groups.

If $A\cong G$, the result is: If $p$ is any element of $G$, then $G$ is isomorphic to itself via the map $a_i\rightleftarrows a_i.p$.
\end{proof}
Due to Axiom \ref{Axiom2}, we can always choose $p$ such that (assuming $a_i,a_k$ fixed) $a_i.p=a_k$. Thus,

\subsubsection*{For any distributive group there exists a unique isomorphism such that a fixed element $a_i$ gets mapped to $a_k$.}

From this, the homogeneity of distributive groups follows.

\subsubsection*{Every group theoretic property, which holds for at least one element in a distributive group, holds for all elements of this group}

\begin{enumerate}[ {Conclusion} 1:]
\item If an element $a_i$ is an element of exactly $h$ different subgroups of order $v$ of a group $G$, then this holds for each element in $G$.

Thus, if $G$, a finite group of order $N$, has $n$ different subgroups of order $v$, namely \begin{equation*}A_1, A_2,\ldots,A_n\end{equation*} then each element of $G$ is an element of exactly $h$ of those subgroups. Hence, \begin{equation*}Nh=nv\end{equation*}
\item Let $G$ be a finite group of order $N$, let $A_1$ be a subgroup of $G$ with order $v$ and let \begin{equation*}A_2,A_3,\ldots,A_m\end{equation*} be all the subgroups of $G$ that are isomorphic to $A_1$. It is easily shown that all groups in the system
\begin{equation}\label{isosubgroups}A_1,A_2,\ldots,A_m\end{equation} are pairwise isomorphic and hence $A_i.p$ is isomorphic to $A_i$ and thus to each $A_j$. Therefore, $A_i.p$ is also an element of this system \eqref{isosubgroups} of isomorphic subgroups.

Now if $a_i$ is an element of $q$ of those groups \eqref{isosubgroups} then this holds for every element of $G$.

Counting all elements of $G$ that are elements of groups in \eqref{isosubgroups}, we get
\begin{equation*}
Nq=mv
\end{equation*}
\end{enumerate}
\subsubsection*{A closer look at the interesting case $q=1$,} for which the group $G$ can be described via its subgroups in the system \eqref{isosubgroups} as 
\begin{equation*}
G=A_1+A_2+\cdots+A_m
\end{equation*}
Now let $A=A_i$, for some $i$, then the union of the \textit{different} subgroups of the union of \begin{equation*}A.a,A.b,\ldots,\text{ etc., where } G=\left\{a,b,\ldots\right\}
\end{equation*} is the same as the groups in \eqref{isosubgroups}.

This leads to the following proposition.
\begin{prop}\label{prop1}
Let $A,B$ be subgroups of $G$ with $B\neq A$, $A\cong B$. If $A\cap B=\emptyset$, then the following holds:
\begin{equation*}
G=A+A.\beta+A.\gamma+\cdots+A.\epsilon
\end{equation*}
Considered as elements, the system of subgroups $A,A.\beta,\ldots$ is itself a distributive group $\Gamma$ which is $v$-step isomorphic\footnote{In the original \textit{$v$-stufig isomorph}. For a definition, see Definition \ref{defn:isomorph}} to $G$.
\end{prop}
To properly understand this proposition, we have to explain the composition of two groups $C\left\{c_1,c_2,\ldots\right\}$ and $D\left\{d_1,d_2,\ldots\right\}$. In particular, $C.D$ is the system of all elements $c_i.d_k$, $i,k=1,2,\ldots$. 
\begin{proof}We now have to show Axiom \ref{Axiom1}, that is $A_i.A_k=A_l$, i.e. a group in the system \eqref{isosubgroups}.

Let $a^r_1,a^r_2,\ldots,a^r_v$ be the elements of the group $A_r$, $r=1,\ldots,m$ and consider the two isomorphic groups
\begin{equation*}a^i_1.A_k\text{ and } A_i.a_l^k
\end{equation*} both of which are elements of the system \eqref{isosubgroups}. Their intersection includes the element $a_1^ia_l^k$ and both must therefore coincide. Thus we have
\begin{equation*}
A_l=a_1^i.A_k=A_i.a_1^k=A_i.a_2^k=\cdots=A_i.A_k
\end{equation*}
and Axiom \ref{Axiom1} holds.

In order to prove Axiom \ref{Axiom2}, we have to show that the equation
\begin{equation*}
A_i.X=A_l
\end{equation*}
has exactly one solution in the system \eqref{isosubgroups}.

Certainly, there is one element $x=a_l^k$ in $G$ such that $a_1^i.a_1^k=a_1^l$. This means, however, that $A_i.A_k=A_l$, that is, $X=A_k$. If there would be a second solution, $A_j$, say, that is $A_i.A_j=A_i=A_k$, this would imply that $a_1^i.A_j=a_1^i.A_k$ and hence $A_j=A_k$. Thus, Axiom \ref{Axiom2} holds. For $Y.A_l=A_k$, the proof is similar.

The proof of Axiom \ref{Axiom3} is straightforward as well: $(A_i.A_k).A_l$ and $(A_i.A_l).(A_k.A_l)$ are groups in \eqref{isosubgroups}. Both groups include the element
\begin{equation*}
(a_1^i.a_1^k).a_1^l=(a_1^i.a_1^l).(a_1^k.a_1^l)
\end{equation*}
and do therefore coincide.
\end{proof}

\subsubsection*{We now give three examples of systems of isomorphic subgroups for which q=1 holds}
\begin{enumerate}[{Example} 1:]
\item Let $G$ be a group of order $N$ and let $A$ be a subgroup of $G$ of order $v$. Let $A$ have the property that the intersection of $A$ with any other subgroup of $G$ of order $v$ be empty. Then $A$ is also disjoint to any subgroup of $A$ which are isomorphic to $A$ and Proposition \ref{prop1} holds.
\item\label{eg2} Let $G$ be a group of order $N$ and let $a_1\in G$. Then the elements $x\in G$, where 
\begin{equation}\label{equation2}a_1.x=x.a_1\end{equation} are elements of a subgroup $A$ of $G$ which is disjoint to any other isomorphic subgroup of $G$.

Since there are only finitely many solutions to equation \eqref{equation2}, namely $x=a_1,a_2,\ldots$, we only have to check Axiom \ref{Axiom1} to show it is a group. In particular, since $a_1.x=x.a_1$ and $a_1.y=y.a_1$, we have that $a_1.(x.y)=(x.y).a_1$ that is, if $x$ and $y$ are a solution, $x.y$ is a solution as well. Let $A=\left\{a_1,a_2,\ldots,a_v\right\}$ be the set of all solutions, then $a_1$ commutes with every other element in $A$. Indeed, for any element $a_i\in A$, $A$ is the subgroup of $G$ of all elements that commute with $a_i$, since any element that commutes with $a_i$ also commutes with $a_1$.

Now let $B\cong A$, another subgroup of $G$. $B$ also has order $v$ and, due to the homogeneity of the group $G$, $B$ includes all elements that commute with each other. Clearly, if $A$ and $B$ are not disjoint, they must be equal, since one element is enough to generate $A$ and $B$, respectively. This proves the statement.
\item First, let us introduce the following notation:
\begin{eqnarray*}
a.(a.b)&=&a^2.b\\
a.\left[a.(a.b)\right] &=&a^3.b
\end{eqnarray*}
Let $G$ be a group of order $N$ and let $a_1\in G$. We claim: All elements $x\in G$ with $a_1^r.x=x$ constitute a subgroup $A$ in $G$ which is disjoint to any other subgroup of $G$ which is isomorphic to $A$. The proof of this being a group is done similar as in Example \ref{eg2}. Let $A$ be the group of solutions $\left\{a_1,a_2,\ldots,a_v\right\}$ and call $r$ the \textit{degree} of the group $A$ of order $v$. Then, since $a_1^r.a_i=a_i$ for $i=1,2,\ldots,v$ and due to the homogeneity of $A$, we have that $a_t^r.a_i=a_i$ for $i,t=0,2,\ldots,v$, that is, $A$ is generated by any one of its elements. Similarly, if $B=\left\{b_1,b_2,\ldots,b_v\right\}$ with $B\cong A$ and $B$ a subgroup of $G$, then $B$ is generated by any one of its elements $b_i$ with $b_i^r.x=x$ \footnote{This is due to the homogeneity of $G$ itself}. Thus, if $A$ and $B$ are not disjoint, $A=B$.
\end{enumerate}
\begin{rem}
Instead of $a_1^r.x=x$ we could have also used $a_1.\left[(x.a_1).a_1\right]=x.a_1$ or indeed any other similar equation, as long as $x$ only occurs once on each side.
\end{rem}
\section{Principles of the generation of distributive groups}\label{chapter3}
\subsection{Generation from Associative Commutative Groups}\label{chapter3.1}
Let $G_\alpha=\left\{a_1,a_2,\ldots,a_{2n+1}\right\}$ be an associative, commutative group of order $2n+1$, let $\alpha\in\mathbb{N}$ and let $\alpha$ and $\alpha-1$ be coprime with $2n+1$ (for example $\alpha=2$ or $\alpha=n+1$). The composition of the two elements $a_i,a_k\in G$ shall be denoted by $a_i\circ a_k$ and then the following holds:
\begin{equation*}
a_i\circ a_k=a_k\circ a_i\text{ and }(a_i\circ a_k)\circ a_l = a_i\circ(a_k\circ a_l)
\end{equation*}
Furthermore, $\alpha_i^\alpha=a_i\circ a_i\circ\ldots\circ a_i$, iterated $\alpha$ times. For all $g\in G_a$, there is an $i\in1,\ldots,2n+1$ such that $a_i^\alpha=g$. This holds since if not, there is some $i,k$ such that $a_i\neq a_k$ with $a_i^\alpha=a_k^\alpha$. Let $a_k^{-1}$ be the inverse element of $a_k$ in $G_a$, then $a_i^\alpha\circ(a_k^{-1})^\alpha=1=(a_i\circ a_k^{-1})^\alpha$. However, $a_i\circ a_k^{-1}$ cannot be the unit element, since otherwise $a_k^{-1}=a_i^{-1}$, that is $a_i=a_k$ which is a contradiction. Now let $a_l=a_i\circ a_k^{-1}$ then $a_l^\alpha=1$ for some $\alpha$, thus there must be a subgroup of $G_a$ of order $v$, where $v$ is a divisor of $\alpha$. However, $v$ would then divide $2n+1$ as well, which, contrary to the assumptions, means that $\alpha$ and $2n+1$ are not coprime. Hence, the equation $x^\alpha=a_i$, i.e. $x=a_i^{\frac{1}{\alpha}}$ is uniquely solvable in $G_a$\footnote{We can always find a $\beta\in\left\{1,\ldots,2n+1\right\}$ such that $a^{\frac{1}{\alpha}}=a^\beta$. In fact, since $a=a^{\alpha\beta}$, $a^{\alpha\beta-1}=1$. Clearly, this holds for $\alpha\beta-1=2n+1$ i.e. for $\alpha\beta\equiv1\mod(2n+1)$ and since $\alpha$ and $2n+1$ are coprime, such a $\beta$ exists.}. The same holds for $\alpha-1$ instead of $\alpha$.

We will now show that we can understand the elements of $G_a$ as the elements of a distributive group $G_a$ of order $2n+1$ if we use the following composition
\begin{equation*}
a_i.a_k=a_i^\alpha\circ a_k^{1-\alpha}
\end{equation*}
Axioms \ref{Axiom1} and \ref{Axiom2} hold true due to our assumptions so we only have to show that Axiom \ref{Axiom3} holds.

We have
\begin{equation*}
(a_i.a_k).a_l=(a_i^\alpha.a_k^{1-\alpha})^\alpha\circ a_l^{1-\alpha}=a_i^{\alpha^2}\circ a_k^{\alpha\left(1-\alpha\right)} \circ a_l^{1-\alpha}
\end{equation*}
as well as
\begin{equation*}
(a_i.a_l).(a_k.a_l)=(a_i^\alpha\circ a_l^{1-\alpha})^\alpha\circ(a_k^\alpha\circ a_l^{1-\alpha})^{1-\alpha}=a_i^{\alpha^2}\circ a_k^{\alpha(1-\alpha)}\circ a_l^{1-\alpha}
\end{equation*}
Thus, Axiom \ref{Axiom3} holds.

\begin{remnmb}
For the special, commutative group $G_a$ of order $2n+1$ whose elements are $1,2,\ldots,2n+1$ with composition $a\circ b\equiv a+b\mod 2n+1$, $a^\alpha=\alpha a$, i.e.
\begin{equation*}
a.b\equiv\alpha a+(1-\alpha)b\mod 2n+1
\end{equation*}
\end{remnmb}
\begin{remnmb}
Let $G_a$ be the infinite commutative group with the elements being $\mathbb{R}$ and whose composition is $a.b=a+b$. Then $a.b=\alpha a+\beta b$ with $\alpha+\beta=1$ is the composition of the respective distributive group.
\end{remnmb}
\begin{remnmb}
Let $A_a$ be a subgroup of $G_a$ of order $v$, then v is coprime to $\alpha$ and $\alpha-1$ since $v$ divides 2n+1. Let $b_1,b_2,\ldots b_v$ be the elements of $A_a$, then this system together with the composition $b_i.b_k=b_i^\alpha\circ b_k^{1-\alpha}$ is a distributive group. Thus, all elements of the subgroup $A_a$ of $G_a$ constitute a subgroup $A_d$ of a group $G_d$.

The elements
\begin{equation*}
A_d.p=\left\{b_1.p,b_2.p,\ldots,b_v.p\right\}=\left\{b_1^\alpha\circ p^{1-\alpha},\ldots,b_v^\alpha\circ p^{1-\alpha}\right\}=\left\{b_1^\alpha,\ldots,b_v^\alpha\right\}\circ p^{1-\alpha}=A^\alpha\circ p^{1-\alpha}
\end{equation*}
of the distributive subgroup $A_d.p$ constitute a "coset"\footnote{In the original \textit{Nebengruppe} as compared to \textit{Untergruppe}, which is a subgroup.} $A_\alpha\circ p^{1-\alpha}$ of the subgroup $A_a$ of $G_a$. Thus, if $A_d.p$  and $A_d.q$ have an element in common, they are in fact equal and the following decomposition holds:
\begin{equation}\label{subgroupgroup}
G_d=A_d+A_d.p+\cdots+A_d.t
\end{equation}
Now let 
\begin{eqnarray*}
b_i.p&=b_i^\alpha\circ p^{1-\alpha}&\in A_d.p\\
b_j.q&=b_j^\alpha\circ q^{1-\alpha}&\in A_d.q
\end{eqnarray*} then
\begin{eqnarray*}
(b_i.p).(b_j.q)&=&(b_i^\alpha\circ p^{1-\alpha})^\alpha\circ(b_j^\alpha\circ q^{1-\alpha})^{1-\alpha}\\
&=&(b_i^\alpha\circ b_j^{1-\alpha})^\alpha\circ(p^\alpha\circ q^{1-\alpha})^{1-\alpha}\\
&=&(b_i.b_j).(p.q)
\end{eqnarray*}
However, since this element is an element of $A_d.(p.q)$, the subgroups in \eqref{subgroupgroup} constitute, taken as elements, another distributive group (Axioms \ref{Axiom2} and \ref{Axiom3} hold with proof as before).

Furthermore, we have the following composition:
\begin{equation*}
(b_i.p).(b_j.q)=(b_i.b_j).(p.q)
\end{equation*}
Let $B_d=\left\{b_1,\ldots,b_\mu\right\}$ and $C_d=\left\{c_1,\ldots,c_r\right\}$ be two subgroups of $G_d$. Then from 
\begin{equation*}
(b_i.c_k).(b_j.c_i)=(b_i.b_j).(c_k.c_i)
\end{equation*}
it follows that the elemental system
\begin{equation*}
\left\{\ldots,(b_i.c_k),\ldots\right\}, i=1.\ldots \mu, j=1,\ldots,r
\end{equation*}
is a group.
\end{remnmb}
\subsection{Generation of Distributive Groups from Distributive Groups}
\begin{rem}
The same holds for associative Groups.
\end{rem}

Let $A=\left\{a_1,\ldots,a_v\right\}$ and $B=\left\{b_1,\ldots,b_\mu\right\}$ with their respective compositions $a_i.a_k$ and $b_i\odot b_k$ be two distributive groups of order $v$ and $\mu$, respectively. Then the elemental system $\left\{\ldots,(a_i,b_k),\ldots\right\}$ with the composition $(a_i,b_k)\times(a_j,b_l)=(a_i.a_j,b_k\odot b_l)$ is also a distributive group and has order $v\mu$. It is immediately clear that Axioms \ref{Axiom1} to \ref{Axiom3} hold.

A group that has been generated like this such that every element $(a_i,b_k)$ is numbered by two indices shall be called \textit{Double Index Group}\footnote{In the original \textit{Zweiindizesgruppe}}. Similarly, triple index groups and higher can be generated.
\begin{rem}
If the composition $a.b=a^\alpha\circ b^{1-\alpha}$ gets replaced by the composition $a.b=a^\alpha\circ b^\beta$, where $\alpha, \beta$ are coprime with $2n+1$, then from the Cayley table it is clear that Axioms \ref{Axiom1} and \ref{Axiom2} hold. Instead of Axiom \ref{Axiom3}, we have
\begin{equation*}
(a.d).(b.d)=(a.b).(d.d)\text{, and } (d.a).(d.b)=(d.d).(a.b)
\end{equation*}
or more generally
\begin{equation*}
(a.b).(c.d)=(a.c).(b.d)
\end{equation*}
For those groups the relationship with abelian groups of odd order as outlined in this chapter hold.

If $A$ is a subgroup, then $A.p$ is a "sidegroup", similarly as for associative groups.
\end{rem}
\section{Subgroups, Index Theorem}\label{chapter4}
\begin{thm}\label{simplethm}
A group is called \textit{simple} if it has no proper subgroup of order bigger than one. It follows that any simple group is completely described by any two of its elements.
\end{thm} 
\begin{proof}
Let $A=\left\{a_1,\ldots,a_v\right\}$ be a simple subgroup of a group $G$ and let $p\in G$. Then we will show that $A.(A.p)$ is an \emph{r-parameter} group $A.t$. \footnote{In the original \textit{r-gliedrige Gruppe}. A definition is provided in Definition \ref{defn:rgliedrig}.}  We denote the group $A.p$ by $B=\left\{a_i.p=b_i, i=1,\ldots,v\right\}$. Then the following holds
\begin{equation*}
a_i.b_1=a_i.(a_1.p)=(a_i.a_1).b_i=a_j.b_i
\end{equation*}
where $a_i.a_1=a_j$.

Let $a_i\neq a_1$, then $a_i\neq a_j$ and also $b_i\neq b_j$. The two groups $a_j.B$ and $A.b_1$ are uniquely isomorphic to $A$, so are simple as well. $a_j.b_1$ and $a_i.b_1=a_j.b_i$ are elements of both groups and so they are equal as simple groups, that is
\begin{equation*}
A.b_1=a_j.B,\text{ with } j=2,3,\ldots,v
\end{equation*}
Similarly,
\begin{equation*}
A.b_2=a_h.B,\text{ with } h=1,3,4,\ldots,v
\end{equation*}
from which it follows that 
\begin{equation*}
a_1.B=a_2.B=\cdots=a_v.b=A.B=A.(A.p)
\end{equation*}
\end{proof}
\subsubsection*{We now want to extend Theorem \ref{simplethm} to groups that are uniquely determined by exactly two of their elements (but not any two elements)}

First, we will show the following lemma.
\begin{lem}
Let there be such a subgroup $A$ and let $a_l, a_k\in A$ the two elements that determine $A$. Then if $a_i.a_j=a_k$, both $a_i$ and $a_j$ as well as $a_k$ and $a_j$ determine the subgroup, too.
\end{lem}
\begin{proof}
$a_i$ and $a_j$ determine a smallest group $\bar{A}$ for which $\bar{A}\leq A$ holds; but $a_k\in \bar{A}$ and thus $A\subset\bar{A}$, that is $\bar{A}\geq A$ and hence $A=\bar{A}$.
\end{proof}
The same proof holds for $a_k$ and $a_j$.
We will now prove the theorem.
\begin{proof}
Let the group $A$ be determined by its elements $a_l$ and $a_k$. Furthermore, as above, we have
\begin{eqnarray}
\label{eq4}&a_i.a_j=a_k& \text{ fixed, then}\\
\label{eq5}&a_k.b_j=(a_k.a_j).b_k & \text{ and}\\
\label{eq6}&a_l.b_j=(a_i.a_j).b_l=a_k.b_l=(a_k.a_l).b_k.
\end{eqnarray}
The uniquely isomorphic groups $A.b_j$ and $A.b_k$ both include the two elements $a_k.b_j=(a_k.a_j).b_k$ and $a_i.b_j=(a_k.b_l).b_j$.

The group $A.b_j$ is uniquely determined by the two elements $a_k.b_j$ and $a_l.b_j$; similarly, $(a_k.a_j).b_k$ and $(a_k.a_l).b_k$ determine $A.b_l$ uniquely. Thus, $A.b_j=A.b_k$.

Let $x_1,x_2,\ldots\in G$ be all the elements in $G$ for which $A.x_1=A.x_2=\cdots A.b_j$ holds, then those elements constitute a subgroup of $G$. Since there are only finitely many such elements, we only have to show that Axiom \ref{Axiom1} holds.
\begin{equation*}
A.(x_p.x_q)\leq (A.x_p).(A.x_q)=(A.x_p).(A.x_p)=A.x_p=A.b_j
\end{equation*}
shows this group property. Since both $b_j$ and $b_k$ are also in this group and hence $B$ is, we have 
\begin{equation*}
A.b_1=A.b_2=\cdots=A.b_v=A.B
\end{equation*}
\end{proof}
\begin{prop}
Let $A=\left\{a_1,a_2,\ldots,a_v\right\}$ be a subgroup with the property that for any two elements $a_i,a_k\in A$, the equation $a_i=a_k^2.a_i$ does not hold. Then $A.(A.p)=A.t$.
\end{prop}
We will use equations \eqref{eq4}, \eqref{eq5} and \eqref{eq6}.
\begin{proof}
Let $k=1,2,\ldots v$, then $a_k.b_j=(a_k.a_j).b_k$ and $a_i.b_j=(a_k.a_i).b_k$ are both elements in $A.b_j$. Thus, their product $a_i.a_k).b_j=\left[a_k.(a_i.a_j)\right]=(a_k.a_k).b_k=a_k.b_k$ is also in $A.b_j$.

Hence, each of the elements
\begin{equation}
a_1.b_1,a_2.b_2,\ldots,a_v.b_v
\label{eq7}
\end{equation}
is an element of each of the groups $A.b_j$.

If all of the elements in \eqref{eq7} are distinct, then they determine uniquely the group $A.b_j$ and $A.b_1=A.b_2=\cdots=A.b_v=A.B$ holds.

Let us take a closer look at 
\begin{equation}a_r.b_r=a_l.b_l\label{eq10}
\end{equation}
 There exists an $a_k\in A$ for which 
\begin{equation}\label{eq8}
a_r.a_k=a_l
\end{equation} holds. Thus,
\begin{eqnarray*}
a_r.b_k&=(a_r.a_k).b_r&=a_l.b_r\\
&=(a_l.a_r).b_l&=(a_l.b_l).(a_r.b_l)\\
&=(a_r.b_r).(a_r.b_l)&=a_r.(b_r.b_l)
\end{eqnarray*}
from which $b_k=b_r.b_l$ follows and hence
\begin{equation}\label{eq9}
a_k=a_r.a_l
\end{equation}
From \eqref{eq8} and \eqref{eq9} we conclude that $a_l=a_r.(a_r.a_l)=a_r^2.a_l$. Similarly, $a_r=a_l^2.a_r$.

However, this is a contradiction to the assumptions, so we have that $A.(A.p)=A.t$.
\end{proof}
\begin{rem}
We will show that \eqref{eq10} follows from $a_l=a_r^2.a_l$.
\end{rem}
\begin{proof}
We set $a_r.a_l=a_k$ which is equivalent to $a_l=a_r.a_k$. We now have
\begin{equation*}
a_r.b_k=a_r.(b_r.b_l)=(a_r.b_r).(a.r.b_l)
\end{equation*}
as well as
\begin{equation*}
a_r.b_k=(a_r.a_k).b_r=a_l.b_r=(a_l.a_r).b_l=(a_l.b_l).(a_r.b_i)
\end{equation*}
Comparing these gives us the result, $a_r.b_r=a_l.b_l$.
\end{proof}
\begin{cor}
Let $A=\left\{a_1,\ldots,a_v\right\}$ be a commutative group without any subgroup of order three, then $A.(A.p)=A.t$.
\end{cor}
This follows directly from the preceding remark.

If there were be a composition $a_j=a_i.(a_i.a_j)$, then from $a_i.a_j=a_k$ it follows that $a_j=a_i.a_k$. Furthermore, $a_j.a_k=(a_i.a_k).(a_i.a_j)=a_i.(a_k.a_j)$ from which $a_j.a_k=a_i$ follows, that is, $a_i$, $a_j$ and $a_k$ constitute a subgroup of order three, a contradiction.

\begin{thm}\label{thm1}
Let $A$ be a subgroup for which $A.(A.p)=A.t$ holds. Then 
\begin{equation*}
G=A+A.p+A.q+\cdots+A.w
\end{equation*}
\end{thm}
We will show this by showing that two subgroups $A.p$ and $A.q$ are either identical or disjoint.
\begin{proof}
Let $c$ be an element in the intersection of those two subgroups. Then,

$A.(A.p)=A.c$ and $A.(A.q)=A.c$, respectively and thus $A.(A.p)=A.(A.q)$ and hence $a_1.(A.p)=a_1.(A.q)$ and thus $A.p=A.q$.
\end{proof}
\begin{rem}
We will now show for later use if we have a simple subgroup of order $v$, $A=\left\{a_1,a_2,\ldots,a_v\right\}$, $(A.p).A$ is also a $v$-parameter group $t.A$.
\end{rem}
\begin{proof}
We have
\begin{equation}\label{eq14}
(a_i.p).a_k=(a_i.a_k).(p.a_k)=a_l.(p.a_k)=(a_l.p).(a_l.a_k)=(a_l.p).a_k
\end{equation}
where
\begin{eqnarray}\label{eq11}
a_i.a_k&=a_l&\text{ and}\\
a_l.a_k&=a_h\label{eq11a}
\end{eqnarray}
Furthermore,
\begin{equation}
(a_l.p).a_h=(a_l.a_h).(p.a_h)=a_s.(p.a_h)=(a_s.p).(a_s.a_h)=(a_s.p).a_t
\end{equation}
where
\begin{eqnarray}
a_l.a_h=a_s&\text{ and }&a_s.a_h=a_t
\end{eqnarray}
For $a_s$ and $a_t$ we calculate
\begin{equation}\label{eq13}
\begin{cases}
a_s=a_l.a_h=(a_i.a_k).(a_l.a_k)=(a_i.a_l).a_k\\
a_t=a_s.a_r=\left[(a_i.a_l).a_k\right].(a_l.a_k)=\left[(a_i.a_l).a_l\right].a_k
\end{cases}
\end{equation}
 and thus the two simple groups $(a_i.p).A$ and $A.p.a_h$ have the following elements in common
\begin{eqnarray}\label{eq12}
(a_i.p).a_h&\text{ and }&(a_i.p).a_k=(a_l.p).a_h
\end{eqnarray}
Let $A_i\neq a_k$, then because of \eqref{eq11} we have that $a_k\neq a_l$. Then, by \eqref{eq11a}, $a_h\neq a_k$, which means the two elements in \eqref{eq12} are different.
This leads to
\begin{equation*}
(a_i.p).A=(A.p).a_h
\end{equation*}
with $a_i.a_k=a_l$, $a_l.a_k=a_h$, that is $(a_i.a_k).a_k=a_h$.

Now there are two possible cases.
\begin{enumerate}[{Case} 1:]
\item\label{case1} For two values $a_k$ ($a_k$ and $a_{\bar{k}}$) let $a_h\neq a_{\bar{h}}$. Then we have 
\begin{equation*}
(A.p).a_h=(A.p).a_{\bar{h}}=(A.p).A
\end{equation*}
\item Let $a_h$ be independent from $a_k$ in $a_h=(a_i.a_k).a_k$ for all $k\neq i$. Then it must be that $a_h=a_i$, since if otherwise $a_h=a_t, t\neq i$, we would have $a_t=(a_i.a_t).a_t$ that is, $a_i.a_t=a_t$ from which it follows that $t=i$ which is a contradiction. Thus, the following must hold:
\begin{equation}\label{eq15}
a_i=(a_i.a_k).a_k
\end{equation}
Since $a_i$ was arbitrary, equation \eqref{eq15} must hold for any two elements in $A$ (otherwise it would be Case \ref{case1}). But then it follows from equation \eqref{eq13} that $a_t=a_i.a_k$ and the two simple groups $(a_i.p).A$ and $(A.p).a_i$ both include the two distinct elements $(a_i.p).a_t$ and $(a_i.p).a_k=(a_s.p).a_t$. (Since $a_t\neq a_k$ if $a_i\neq a_k$.) Thus, it follows that
\begin{equation}\label{eq16}
(a_i.p).A=(A.p).a_t
\end{equation}
If $k$ is running through $1,2,\ldots,i-1,i+1,\ldots,v$, then $t$ runs through the same numbers (possibly in changed order), that is
\begin{equation*}
(a_i.p).A=(A.p).a_1=(A.p).a_2=\cdots=(A.p).a_v
\end{equation*}
since $a_i$ was arbitrary.
\end{enumerate}
\end{proof}
\begin{prop}
Let $A=\left\{a_1,\ldots,a_v\right\}$ be a subgroup with the property that for any two elements the following equations do not hold:
\begin{eqnarray*}
(a_i.a_p).a_p&=&a_i\\
(a_p.a_i).a_p&=&a_i
\end{eqnarray*}
Then $(A.p).A=t.A$.
\end{prop}
\begin{proof}
We are following equations \eqref{eq14} to \eqref{eq13}.

The group $(a_i.p).A$ includes the elements $(a_i.p).a_h$ and $(a_i.p).a_k=(a_l.p).a_h$ as well as all of the elements $(x_r.p).a_h$ where $x_r$ is any element of the smallest group with elements $a_i,a_l$, $\left\{a_i,a_l\right\}$. This group also includes the elements $a_k,a_h,a_s,a_t$ and so $(a_h.p).a_h\in (a_i.p).A$ with $a_h=(a_i.a_k).a_k$.

We claim that, together with $a_k$, $a_h$ runs through all elements of $A$. From $(a_i.a_k).a_k=(a_i.a_p).a_p=a_r$ we conclude as follows:

Set $a_i.a_k=a_l$ and $a_i.a_p=a_q$, then $a_l.a_k=a_q.a_p$; furthermore $a_l.(a_p.a_k)=(a_i.a_k).(a_p.a_k)=(a_q.a_k)$. In addition,
\begin{eqnarray*}
a_i.(a_p.a_k)&=(a_i.a_p).(a_l.a_k)&=(a_l.a_p).(a_q.a_p)\\
&=(a_l.a_q).a_p&=\left[a_i.(a_k.a_p)\right].a_p\\
&=a_q.\left[(a_k.a_p).a_p\right]
\end{eqnarray*}
Comparing gives $a_k=(a_k.a_p).a_p$ which is a contradiction. Hence, $(a_i.p).A$ includes all elements
\begin{equation}\label{eq17}
(a_1.p).a_1, (a_2.p).a_2,\ldots,(a_v.p).a_v
\end{equation}

Next, set $b_t=a_t.p$, then $b_l.a_l=b_k.a_k$.

If $a_r.a_k=a_l$, then we form
\begin{eqnarray*}
b_r.a_k&=(a_r.p).a_k&=a_l.(p.a_k)\\
&=b_l.(a_l.a_k)&=(b_l.a_l).(b_l.a_k)\\
&=(b_k.a_k).(b_l.a_l)&=(b_k.b_l).a_k
\end{eqnarray*}
From this, it would follow that $b_r=b_k.b_l$, that is $a_r=a_k.a_l$ or, equivalently, $(a_k.a_l).a_k=a_l$, again a contradiction. Thus, all groups $(a_i.p).A$ include the different elements in \eqref{eq17} and the following holds:
\begin{equation*}
(a_1.p).A=(a_2.p).A=\cdots=(a_v.p).A=(A.p).A
\end{equation*}
\end{proof}
\subsubsection*{Implications from Theorem \ref{thm1}}
\begin{enumerate}[1.]
\item A group whose order is prime, is simple.
\item The order of a simple subgroup, the order of a subgroup that is uniquely determined by two elements and the order of a subgroup, where the equation $a_i^2.a_k=a_k$ does not hold for any pair of elements, divides the order of the group.
\item For subgroups of a commutative group, whose order is not divisible by 3, the index theorem holds (the order of the subgroup is not a divider of the order of the group).
\end{enumerate}
To finish this section, we note the following theorem.
\begin{thm}
If $(A.p).A=t.a$, then also $A.(p.A)=A.t$, since it is always true that $(A.p).A=A.(p.A)$.
\end{thm}
\begin{proof}
Let $(a_i.p).a_k\in(A.p).A$, then since
\begin{equation*}
(a_i.p).a_k=(a_i.a_k).(p-a_k)
\end{equation*}
it is also an element of $A.(p.A)$ and hence
\begin{equation*}
(A.p).A\leq A.(p.A)
\end{equation*}
Let $a_i.(p.a_k)=(a_i.p).(a_i.a_k)\in A.(p.A)$, then it is also an element of $(A.p).A$ and hence
\begin{equation*}
(A.p).A\geq A.(p.A)
\end{equation*}
and thus 
\begin{equation*}
(A.p).A=A.(p.A)
\end{equation*}
\end{proof}
\begin{rem}
Let $G=\left\{a_1,a_2,\ldots,a_v\right\}$ be a distributive group where the equation $(a_i.a_k).a_k=a_i$ does not hold for any $a_i,a_k\in G$. Then we define a new composition,
\begin{equation*}
a_i\circ a_k=(a_i.a_k).a_k
\end{equation*}
For this new system, Axioms \ref{Axiom1} and \ref{Axiom2} hold as well as
\begin{eqnarray*}
(a_i\circ a_k)\circ(a_j\circ a_k)&=\left[(a_i.a_k).a_k\right]\circ\left[(a_j.a_k).a_k\right]&=\left\{\left[(a_i.a_k).a_k\right].\left[(a_j.a_k).a_k\right]\right\}.\left[(a_j.a_k).a_k\right]\\
&=\left\{\left[(a_i.a_j).a_k\right].a_k\right\}.\left[(a_j.a_k).a_k\right]&=\left\{\left[(a_i.a_j).a_j\right].a_k\right\}.a_k\\
&=(a_i\circ a_j)\circ a_k&
\end{eqnarray*}
On the other hand, the left-sided distributivity does generally not hold.
\end{rem}
\section{The Structure of Distributive Groups}
In this section, we assume that for each subgroup $A$ of $G$ $A.(A.p)$ and $(A.p).A$ are subgroups of the same order as $A$. Equivalently,
\begin{equation*}
\begin{cases}
A.(A.p)=A.t\\
(A.p).A=A.(p.A)=A.\tau
\end{cases}
\end{equation*}
As shown in section\ref{chapter4}, a sufficient condition for this to hold is that for any $x,y\in G$, neither of the following two equations hold:
\begin{equation*}
(x.y).y=x,\ (y.x).y=x,\ y.(y.x)=x
\end{equation*}
We will call those groups \textit{distinguished groups}\footnote{In the original \textit{ausgezeichnete Gruppe}}.
\begin{notation}
A subgroup $A.p$ is denoted by $\left[A\right]$
\end{notation}
\begin{lem}\label{lem1}
If $A$, $B$ and $C$ are subgroups, all of order $v$, then if $A.B=C$,
\begin{equation*}
B=\left[A\right],\ C=\left[A\right],\ A=\left[B\right],\ C=\left[B\right],\ A=\left[C\right],\ B=\left[C\right]
\end{equation*}
\end{lem}
\begin{proof}
Since $A.B=C$, $C=\left[A\right]$. Let $A=(a_t)$, $B=(b_t)$, $C=(c_t)$, with $t=1,2,\ldots,v$ be all the elements of the three subgroups.

There exists an element $r\in G$ such that $a_1.r=b_1$. Then we have
\begin{equation*}
C=A.B=A.b_1=A.(a_1.r)=A.(A.r)
\end{equation*}
that is, $B=\left[A\right]$. Furthermore, $a_1=b_1.q$ for some $q\in G$. Then
\begin{equation*}
C=A.B=a_1.B=(b_1.q).B=(B.q).B=B.(q.B)
\end{equation*}
and hence $C=\left[B\right]$ and $A=\left[B\right]$.

Finally, $A.C=A.\left[C\right]=A.\sigma$, and thus, similarly as before, $A=\left[C\right]$ and in the same way, from $B.C=B.\left[B\right]=B.\tau$ it follows that $B=\left[C\right]$.
\end{proof}
\begin{lem}\label{lem2}
Let $A$, $B$ and $C$ be $v$-parameter subgroups of, then since $A.B$ is a $v$-parameter subgroup, $A=\left[C\right]$ and $B=\left[C\right]$, it follows that $A.B=\left[C\right]$.
\end{lem}
\begin{proof}
In fact, since $A=C.\theta$ and $B=C.\tau$, 
\begin{equation*}
A.B=(C.\sigma).(C.\tau)\geq C.(\sigma.\tau)
\end{equation*}
But since $A.B$ has parameter $v$, it follows that $A.B=C.(\sigma.\tau)=\left[C\right]$.
\end{proof}
\begin{lem}\label{lem3}
Let $A_1,A_2,\ldots,A_\sigma$ be $v$-parameter subgroups with the property that pairwise products $A_i.A_k,\ i,k=(1,\ldots,\sigma)$ are also $v$-parameter groups. Looking at the span of all those subgroups with pairwise products, one gets the system $A_1,A_2,\ldots,A_{\sigma+\tau}$ with the property that the pairwise products $A_i.A_k,\ i,k=1,2,\ldots,\sigma+\tau$ is also a $v$-parameter group.
\end{lem}
\begin{proof}
We only have to show that $A_i.A_{\sigma+k}$, $A_{\sigma+k}.A_i$ and $A_{\sigma+h}.A_{\sigma+k}$, $i=1,2,\ldots,\sigma$, $h,k=1,2,\ldots,\tau$ are $v$-parameter groups.

Since $A_i.A_p$ are a $v$-parameter group, it follows by Lemma \ref{lem1} that $A_i=\left[A_p\right],\ i,p=1,\ldots,\sigma$ and since $A_i.A_k$ is a $v$-parameter group, it follows from Lemma \ref{lem2} that
\begin{equation*}
A_i.A_k=\left[A_p\right],\ p,i,k=1,2,\ldots,\sigma
\end{equation*}
Thus, $A_{\sigma+k}=\left[A_i\right]$, since $A_{\sigma+k}=A_r.A_s$ by the assumption, and hence $A_i.A_{\sigma+k}$ (and also $A_{\sigma+k}.A_i$ is a $v$-parameter group.

From this we deduce that $A_i=\left[A_{\sigma+k}\right],\ i=1,\ldots,\sigma,\ k=1,\ldots,\tau$. Now let $A_{\sigma+h}=A_p.A_q$, then $A_p=\left[A_{\sigma+k}\right]$ and $A_q=\left[A_{\sigma+k}\right]$ and since $A_p.A_q$ is a $v$-parameter group, by Lemma \ref{lem2} $A_p.A_q=A_{\sigma+h}=\left[A_{\sigma+k}\right]$ and thus indeed $A_{\sigma+h}.A_{\sigma+h},\ h,k=1,2,\ldots,\tau$ are groups of parameter $v$.
\end{proof}
Now we can proof the following theorem.
\begin{thm}\label{thm2}
Let $A_1,A_2,\ldots,A_\sigma$ be groups of parameter $v$ with the property that any product $A_i.A_k,\ i=1,2,\ldots,\sigma$ is also $v$-parameter, then we can generate a distributive group, where the elements are $v$-parameter subgroups. In addition, $A_1,A_2,\ldots,A_\sigma$ are elements in this group.
\end{thm}
\begin{proof}
We complete the system $A_1,A_2,\ldots,A_\sigma$ with those products $A_i.A_k$ that are not included and get a new system $A_1,\ldots,A_{\sigma+\tau}$ that will be completed in the same way and so on until we arrive at a system 
\begin{equation}\label{eq18}
A_1,A_2,\ldots,A_\epsilon
\end{equation}
which is closed under the products $A_i.A_k$. Since $G$ is finite, such a system exists. In particular, since all groups $A_k,\ k=1,2,\ldots,\epsilon$ will have the form $A_k=\left[A_1\right]$ since $A_i.A_k$ is $v$-parameter and are therefore disjoint\footnote{see Theorem \ref{simplethm}}. Thus, Axiom \ref{Axiom1} holds for the system \eqref{eq18}.

Now let $A_i$ be a subgroup of \eqref{eq18}, then 
\begin{equation*}
A_i.A_k\text{ and } A_k.A_i\text{, respectively, with } k=1,2,\ldots,\epsilon
\end{equation*}
runs through all elements of the system \eqref{eq18} since from
\begin{equation*}
A_i.A_k=A_i.A_j
\end{equation*}
it follows that 
\begin{equation*}
A_1^i.A_k=A_1^i.A_j
\end{equation*}
and hence $A_k=A_j$, which proves that Axiom \ref{Axiom2} holds.

Finally, we will show one of the two relations of Axiom \ref{Axiom3}, e.g.
\begin{equation*}
(A_i.A_k).A_j=(A_i.A_j).(A_k.A_j)
\end{equation*}
The groups $(A_i.A_k).A_j$ and $(A_i.A_j).(A_k.A_j)$ are elements of the system \eqref{eq18} by Axiom \ref{Axiom1} and they are in fact identical, since both include the element $(a_1^i.a_1^k).a_1^j=(a_1^i.a_1^j).(a_1^k.a_1.j)$, which proves the Theorem.
\end{proof}
As an application, we add the following corollary.
\begin{cor}
One can always build a distributive group from $A$ and $A.p$, whose elements are subgroups of the form $A.h$ and of which $A$ and $A.p$ are elements.
\end{cor}
In fact, $A_1=A$ and $A_2=A.p$ fulfil the assumptions of Theorem \ref{thm2}.
\begin{defn}
Let $G$ be a group, then we call a subgroup $A$ of $G$ a \textit{maximal subgroup} if $G$ is completely described by the elements $a_1,a_2,\ldots,a_v$ of $A$ and another element $p\in G-A$ in the complement of $A$ in $G$. We denote this by $\left\{A,p\right\}$. Equivalently, $G$ is the smallest group such that $a_1,a_2,\ldots,a_v,p\in G$.
\end{defn}
Then $G\geq A+A.q_1+A.q_2+\cdots+A.q_\sigma$, $p\in A.q_1$ and since $\left\{A,p\right\}$ must be the same $G$ we have $G=A+Aq_1+\cdots+Aq_\sigma$.

Now set $A=G_1$, then the analogous holds for the maximal subgroup\footnote{In the original \textit{Maximalteiler}.} $A_1$ of $G$:
\begin{equation*}
A=A_1+A_1.r_1+A_1.r_2+\cdots+A_1.r_l
\end{equation*}
We can keep reducing like this, until we arrive at a simple group (which does not have a divider and hence no maximal subgroup). This is because every finite group must have a maximal subgroup, which is of order 1 for simple groups.

To show this, let $A$ be a subgroup of $G$ and let $\left\{A,p\right\}\neq G$ \footnote{$\left\{A,p\right\}$ is the smallest subgroup of $G$ with $A$ and $p$ as elements.}. Then $\left\{A,p\right\}=A_1>A$ is a proper subgroup of $G$. Then $\left\{A_1,p_1\right\}$ is either $G$ or again a proper divider $A_2$ of $G$, and so on.

Since $G$ is finite, this must eventually end.

\subsection*{Addendum 1}\footnote{We did not manage to prove that$(A.p).(A.q)$ is a group of the same order as $A$. We do not know if this is correct or not. The theorems of the Addenda were proved during our attempts to prove this conjecture.}
We denote the complex
\begin{equation}\label{eq19}
(A.p).(A.q)
\end{equation}
where $A$ is a simple group, with $H_{pq}$. All the following theorems for the complex $H_{pq}$ also hold in the case $H_{pq}=A.(p.q)$.
\begin{thm}\label{thm3}
$H_{pq}$ has either $v$ or $v^2$ elements, where $v$ is the order of $A$.
\end{thm}
\begin{proof}
We will show that if $H_{pq}$ has $t<v^2$ elements then $t=v$. If $t<v^2$, then there are two elements $(a_h.p).(a_k.q)$ and $(a_{\bar{h}}.p).(a_{\bar{k}}.q)$, where $h\neq\bar{h}$ and $k\neq\bar{k}$, of the form
\begin{equation}\label{eq20}
(a_h.p).(a_k.q)=(a_{\bar{h}}.p).(a_{\bar{k}}.q)
\end{equation}
But then by Theorem \ref{simplethm}
\begin{equation}\label{eq21}
(A.p).(a_k.q)=(A.p).(a_{\bar{k}}.q)
\end{equation}
holds. But since $k\neq\bar{k}$ and since $A.p$ is uniquely determined by two elements, it follows from \eqref{eq21} that
\begin{equation}\label{eq22}
(A.p).(a_1.q)=(A.p).(a_2.q)=\cdots=(A.p).(A.q)=A.(p.q)
\end{equation}
Hence, in this case $t=v$.
\end{proof}
\begin{rem}
If $H_{pq}$ has $v$ elements, then $H_{pq}=A.(p.q)$. If $H_{pq}$ has $v^2$ elements, then
\begin{equation}\label{eq23}
H_{p1}=(A.p).(a_1.q)+(A.p).(a_2.q)+\cdots+(A.p).(a_v.q)
\end{equation}
\end{rem}
\begin{thm}
$H_{pq}$ is a group.
\end{thm}
\begin{proof}
Assume $H_{pq}$ has $v$ elements, then $H_{pq}=A.(p.q)$ and is therefore a $v$-parameter group. Thus, we only have to show that $H_{pq}$ is a group if it has $v^2$ elements. From equation \eqref{eq23} if follows that
\begin{equation}\label{eq24}
H_{pq}=(A.p).(a_1.q)+\cdots+(A.p).(a_v.q)
\end{equation}
We denote the complex
\begin{equation}\label{eq25}
\left[(A.p).(a_r.q)\right].\left[(A.p).(a_s.q)\right]
\end{equation}
with $K_{rs}$ and will show that it has less than $v^2$ elements from which Theorem \ref{thm3} implies it is a $v$-parameter group of the form
\begin{equation}\label{eq26}
K_{rs}=(A.p).\left[(a_r.a_s).q\right]=(A.p).(a_t.q)
\end{equation} where
\begin{equation}\label{eq26a}
a_r.a_s=a_t
\end{equation}
On the one hand,
\begin{equation}\label{eq27}
\left[(a_t.p).(a_r.q)\right].\left[(a_t.p).(a_s.q)\right]=(a_t.p).\left[(a_r.a_s).q\right]=(a_t.p).(a_t.q))=q_t.(p.q)
\end{equation}
but on the other hand,
\begin{equation}\label{eq28}
\left[(a_r.p).(a_r.q)\right].\left[(a_s.p).(a_s.q)\right]=(a_r.a_s).(p.q)=a_t.(p.q)
\end{equation}
From \eqref{eq27} and \eqref{eq28} it follows that $K_{rs}$ has less than $v^2$ elements, i.e. it is of the form \eqref{eq26}. From \eqref{eq26} and \eqref{eq23},
\begin{equation}\label{eq29}
H_{pq}>K_{rs}
\end{equation}
i.e. $H_{pq}$ is a group.
\end{proof}
\begin{thm}
For the group $H_{pq}$ not only \eqref{eq23} holds but also
\begin{equation}\label{eq30}
H_{pq}=A.\tau_1+A.\tau_2+\cdots+A.\tau_v
\end{equation}
where $\tau_1=p.q$.
\end{thm}
\begin{proof}
We have $A.p=A.(A.r)$ and $A.q=A.(A.s)$ and hence
\begin{equation}\label{eq31}
(A.p=.(A.q)=\left[A.(A.r)\right].\left[A.(A.s)\right]\geq A.\left[(A.r).(A.s)\right]
\end{equation}
But since
\begin{equation}\label{eq32}
\left[a_l.(a_h.r)\right].\left[a_{\bar{l}}.(a_{\bar{h}}.s)\right]=\left[a_l.(a_h.r)\right].\left[a_l.(a_{\bar{h}}.s)\right] \footnote{There is always an element $a_{\bar{\bar{h}}}.s$ in all groups $A.s$ for which $a_{\bar{l}}.(a_{\bar{h}}.s)=a_l.(a_{\bar{\bar{h}}}.s)$ holds, since $A.(A.s)$ is a $v$-parameter group}=a_{\bar{h}}.\left[(a_h.r).(a_{\bar{h}}.s)\right]
\end{equation}
it follows that
\begin{equation}\label{eq33}
\left[A.(A.r)\right].\left[A.(A.s)\right]\leq A.\left[(A.r).(A.s)\right]
\end{equation}
From \eqref{eq31} and \eqref{eq33} we get
\begin{equation*}
H_{pq}=(A.p).(A.q)=A.\left[(A.r).(A.s)\right]
\end{equation*}
and
\begin{equation}\label{eq34}
H_{pq}=A.\sigma_1+\cdots+A.\sigma_{v^2}=A.\tau_1+\cdots+A.\tau_v
\end{equation}
which is convincing once one chooses elements $\sigma_i$ from $(A.r).(A.s)$ one after another and takes into account that $H_{pq}$ only has $v^2$ elements.
\end{proof}
\begin{thm}
If for a certain $k$ in \eqref{eq30} and a certain $l$ in \eqref{eq23} the following relation
\begin{equation}\label{eq35}
A.\tau_k=(A.p).(a_l.q)
\end{equation}
holds, then $H_{pq}$ is of parameter $v$ and vice versa.
\end{thm}
\begin{proof}
We will first prove the first part of the theorem. For this, call the intersection of the two complexes $A$ and $B$ with $\vartheta(A.B)$. Let $H_{pq}$ be of parameter $v$, then for $h=1,2,\ldots,v$
\begin{equation}\label{eq36}
\vartheta(A.\tau_1,(A.p).(a_h.q))\geq a_h.(p.q)
\end{equation}
since $\tau_1=p.q$. Furthermore, $A.\tau_k\neq A.\tau_1$ since $\vartheta(A.\tau_k,(A.p).(a_h.q))=0$ for $h\neq l$. 

However, since $A.\tau_k=(A.p).(a_l.q)$, then by \eqref{eq36} we would have 
\begin{equation}\label{eq37}
\vartheta(A.\tau_1,A.\tau_k)\geq a_l.(p.q)
\end{equation}
This is a contradiction though, since all $A\tau_i$ are pairwise disjoint for $i=1,\ldots,v$ (if $H_{pq}$ has class $v^2$). Thus, $H_{pq}$ must be $v$-parameter.

The inverse is obvious\footnote{In the original \textit{evident}.}.
\end{proof}
\begin{cor}
If $H_{pq}$ has $v^2$ elements, then two subgroups $A.\tau_k$ and $(A.p).(a_l.q)$ have exactly one element in common.
\end{cor}
This follows from the two decompositions \eqref{eq23} and \eqref{eq30} since if they had two elements in common, both would be equal as simple groups and $H_{pq}$ would be $v$-parameter.
\begin{thm}\label{thm4}
From the relation
\begin{equation}\label{eq38}
\vartheta(H_{pq},H_{pr})\neq0
\end{equation}
it follows that $H_{pq}=H_{pr}$.
\end{thm}
\begin{proof}
For $H_{pq}$ and $H_{pr}$ the following decompositions hold
\begin{equation}\label{eq39a}
H_{pq}=A.\tau_1+\cdots+A.\tau_v\text{ and } H_{pr}=A.\sigma_1+\cdots+A.\sigma_v\text{, respectively.}
\end{equation}
(where $A.\tau_i$ are not necessarily all distinct. The same holds for $A.\sigma_i$). So, by the assumption
\begin{equation}\label{eq39}
(a_h.p).(a_k.q)=(a_{\bar{h}}.p).(a_{\bar{k}}.r)
\end{equation}
from which we get
\begin{equation}\label{eq40}
(A.p).(a_k.q)=(A.p).(a_{\bar{k}}.r)
\end{equation}
We will now show that any group $A.\tau_i$ from $H_{pq}$ is identical with a group $A.\sigma_j$ from $H_{pr}$ if and only if $H_{pq}=H_{pr}$.

Let $A.\tau_i$ be this group, so it has at least one element $x$ in common with $(A.p).(a_k.q)$, and so, by \eqref{eq40}, also by $(A.p).(a_{\bar{k}}.r)$. On the other hand, the latter group does have this element $x$ in common with another $A.\sigma_j$ from $H_{pr}$.

Thus, $A.\tau_i$ and $A.\sigma_i$ have the element $x$ in common and are therefore equal.
\end{proof}
\begin{cor}
If an element $a_i\in A$, $A$ a group, is also an element in $(A.p).l$, then $(A.p).l=A$.
\end{cor}
\begin{proof}
In fact, since $A.(A.p)=A.r$, by Lemma \ref{lem2}, $A=(A.p).s$. The two complexes $A.\left[(A.p).h\right]=\left[(A.p).s\right].\left[(A.p).l\right]$ and $A.A=\left[(A.p).s\right].\left[(A.p).s\right]=A$ have the element $a_i$ in common and are therefore identical. Since $A.A$ has parameter $v$, the complex $A.\left[(A.p).l\right]$ must also have parameter $v$ and the following holds
\begin{equation*}
A.\left[(A.p).l\right]=A.A=A
\end{equation*}
i.e.
\begin{equation*}
(A.p).l=A
\end{equation*}
\end{proof}
\begin{thm}
For the subgroups in \eqref{eq30}, the following relation holds
\begin{equation}\label{eq41}
(A.\tau_k).(A.\tau_l)=A.(\tau_k.\tau_l)
\end{equation}
\end{thm}
\begin{proof}
Since $H_{pq}$ is a group, $(A.\tau_k).(A.\tau_l)<H_{pq}$. If $A.\tau_k).(A.\tau_l)$ would be $v^2$-parameter, then
\begin{equation}\label{eq42}
(A.\tau_k).(A.\tau_l)=H_{pq}
\end{equation}
 but since $(A.\tau_k).(A.\tau_k)=(A.\tau_k)<H_{pq}$ we would have
\begin{equation}\label{eq43}
\vartheta\left[(A.\tau_k).(A.\tau_l),(A.\tau_k).(A.\tau_k)\right]=A.\tau_k
\end{equation}
so by Theorem \ref{thm4}
\begin{equation}\label{eq44}
(A.\tau_k).(A.\tau_l)=A.\tau_k
\end{equation}
that is $(A.\tau_k).(A.\tau_l)$ must be $v$-parameter which is a contradiction with the assumption that $(A.\tau_k).(A.\tau_l)$ is $v^2$-parameter. Thus, $(A.\tau_k).(A.\tau_l)$ must be $v$-parameter, i.e. equal to $A.(\tau_k.\tau_l)$.
\end{proof}
Let $G$ be a  group and $A$ be a simple subgroup of $G$, then the composition 
\begin{equation}\label{eq45}
G=A.l_0+A.l_1+\cdots+A.l_s\text{, with } l_0<A
\end{equation}
holds. We remind us that $A.(A.p)=A.t$ holds. Generally, $(A.l_j).(A.p)$ is not necessarily $v^2$-parameter. The groups $A.l_j$ with the property
\begin{equation}\label{eq46}
(A.l_j).(A.l_h)=A.(l_j.l_h)\text{ for each }h=0,1,\ldots,s
\end{equation}
of which $A$ is one of them shall be denoted by $\widetilde{A.l_j}$. Now let
\begin{equation}\label{eq47}
\widetilde{A.l_0}=A,\ \widetilde{A.l_1},\ldots,\widetilde{A.l_r}
\end{equation}
be the entirety of all those groups, then for them the following theorem holds.
\begin{thm}
The groups in \eqref{eq47} are elements of a distributive group $\Gamma$. Let $\Sigma$ be the subgroup of $G$ which includes all elements of $G$ that are present in elements of $\Gamma$.
\end{thm}
\begin{proof}
We will show
\begin{equation}\label{eq48}
(\widetilde{A.l_k}).(\widetilde{A.l_h})=\widetilde{A.(l_k.l_h)}
\end{equation}
By the assumptions, we have
\begin{equation}\label{eq49}
\begin{cases}
(\widetilde{A.l_k}).(A.l_x)=A.(l_k.l_x)\\
(\widetilde{A.l_h}).(A.l_x)=A.(l_h.l_x)
\end{cases}
\end{equation}
so in particular,
\begin{equation}\label{eq50}
(\widetilde{A.l_k}).(\widetilde{A.l_h})=A.(l_k.l_h)
\end{equation}
From \eqref{eq49} we deduce that $\widetilde{A.l_k}=(A.l_x).p$ and $\widetilde{A.l_h}=(A.l_x).q$ and thus $(\widetilde{A.l_k}).(\widetilde{A.l_h})=(A.l_x).(p-q)$. We next have to show $\widetilde{A.(l_k.l_h)}$, i.e. $\left[A.(l_k.l_h)\right].(A.l_x)$ is $v$-parameter. In fact, $\left[A.(l_k.l_h)\right].(A.l_x)=\left[(A.l_x).(p.q)\right].(A.l_x)$ is also $v$-parameter.
\end{proof}
\begin{thm}\label{thm8}
Let $G$ be a distinguished group, $A$ a simple group of $G$ of order $v$ and let $R$ be a subgroup of $G$ which includes $A$ of order $v^2$. If
\begin{equation}\label{eq51}
R=A.l_1+A.l_2+\cdots+A.l_v\text{, with } l_1<A
\end{equation}
is a composition via the group $A$, then $A.l_h=\widetilde{A.l_h}$, $h=1,\ldots,v$.
\end{thm}
\begin{proof}
Let
\begin{equation}\label{eq52}
G=R+R.p_1+\cdots+R.p_t
\end{equation}
From \eqref{eq51} it follows that
\begin{equation}\label{eq53}
R.p_k=(A.l_1).p_k+\cdots+(A.l_v).p_k=(A.p_k).(l_1.p_k)+\cdots+(A.p_k).(l_v.p_k),\ k=1,\ldots,t
\end{equation}
We will show that $R.p_k$ is of the form
\begin{equation}\label{eq54}
R.p_k=A.s_1^{(k)}+\cdots+A.s_v^{(k)}
\end{equation}
Since $A.(A.p_k)$ is $v$-parameter, by Lemmata \ref{lem1} and \ref{lem2}, $A=(A.p_k).\sigma$. Looking at the two complexes
\begin{equation}\label{eq55}
A.(A.p_k)=\left[(A.p_k).\sigma\right].(A.p_k)\text{ and }A.\left[(A.p_k).(l_r.p_k)\right]=\left[(A.p_k).\sigma\right].\left[(A.p_k).(l_r.p_k)\right],\ r=1,\ldots,v
\end{equation}
They are both in the group $R.(R.p_k)=R.p_l$, therefore, by Theorem \ref{thm4}, they are either equal or both $v$-parameter. But since the first complex is $v$-parameter, the second one is as well and the group $(A.p_k).(l_r.p_k)=(A.l_r).p_k$ is of the form $A.s_v^{(k)}$, i.e. we have
\begin{equation*}
R.p_k=A.s_1^{(k)}+\cdots+A.s_v^{(k)}
\end{equation*}
If $A.\tau$ is a subgroup of $G$., then by \eqref{eq52} and \eqref{eq54} it is equal to a subgroup $A.s_r^{(k)}$. By looking at the two complexes
\begin{equation}\label{eq56}
A.(A.s_r^{(k)})\text{ and } (A.l_k).(A.s_r^{(k)}),\ k=1,\ldots, v
\end{equation}
one can deduce in the same way as before that since $A.(A.s_r^{(k)})$ is $v$-parameter, $(A.l_k).(A.s_r^{(k)})$ $v$-parameter so that $A.l_k=\widetilde{A.l_k}$.
\end{proof}
\begin{thm}
Every distinguished group $G$ is either identical with $\Sigma$\footnote{i.e. every subgroup $A.p$ of the group $G$ is a $\widetilde{A.p}$ subgroup} or includes at least $v$ subgroups $\widetilde{A.l_k}$.
\end{thm}
\begin{proof}
If for each $(A.p)$ and $(A.q)$,
\begin{equation}\label{eq57}
(A.p).(A.q)=A.(p.q)
\end{equation}
then $G=\Sigma$.
If this is not the case, then there is at least one subgroup $A.p$ so that $(A.p).(A.q)=H_{pq}$ is $v^2$-parameter. Let $l\in G$ such that
\begin{equation}\label{eq58}
\left[a_1.(p.q)\right].l=a_i\ a_1,a_i\in A
\end{equation}
The group $(H_{pq}).l$ includes $A$, by the Corollary to Theorem \ref{thm4}, i.e.
\begin{equation}\label{eq59}
(H_{pq}.l)=A+A.\sigma_1+\cdots+A.\sigma_{v-1}
\end{equation}
Thus, $(H_{pq}).l$ has the properties of the group $R$ of Theorem \ref{thm8}, that is
\begin{equation*}
A.\sigma_k=\widetilde{A.\sigma_k},\ k=1,2,\ldots v
\end{equation*}
\end{proof}
\begin{rem}
If $G$ is a distinguished group of order $N$, with $N=\prod_{i=1}^np_i$ where $p_i$ are prime numbers with $p_i\neq p_k$ for $i\neq k$, then for $G$ the following relation holds for $G$
\begin{equation}\label{eq60}
(A,p).(A.q)=A.(p.q)
\end{equation}
for all simple subgroups $A$ of $G$.

In fact, if that would not hold true, then $(A.p).(A.q)=H_{pq}$ be a $v^2$-parameter group. Since $G$ is a distinguished group, then by Theorem \ref{simplethm} $v^2$ must divide $N$ which is a contradiction. Thus, \eqref{eq60} holds. But if \eqref{eq60} holds for all simple subgroups of $G$, then it holds for all subgroups of $G$, which is easy to show.

Since every symmetric distributive group $G$ of order $N=\prod_{i=1}^np_i$, with $N$ not divisible by $3$, is a distinguished group, \eqref{eq60} holds for all symmetric groups.
\end{rem}
\subsection*{Addendum 2}
\subsubsection*{A bit about the structure of distributive groups}
Let $G=\left\{a_1,a_2,\ldots,a_v\right\}$ be a distributive group, then from any two elements $a_1$ and $a_2$ we can create the following $l$-cycle (\textit{left} cycle) $a_1.a_2=a_3, a_1.a_3=a_4,\ldots,a_1.a_{h-1}=a_h$. All $a_1,a_2,\ldots,a_h$ shall be different while $a_1.a_h$ should be equal to one of $a_1,a_2,\ldots,a_h$.

$a_1.a_h$ must be different to both $a_1$ and $a_h$, so if $a_1.a_h=a_i$, $i>1$ and we claim that $i=2$; if $i>2$, then from $a_1.a_h=a_i$, $a_h=a_1.a_{h-1}$ and $a_i=a_i.a_{i-1}$ it would follow that $a_1.a_{h-1}=a_{i-1}$; thus, already $a_1.a_{h-1}=a_h=a_{i-1}$ would be equal to one of $a_1,a_2,\ldots,a_{h-1}$ which is a contradiction.

Thus, the cycle looks as follows:
\begin{equation}\label{eq61}
a_1,a_2,a_3=a_1.a_2,a_4=a_1.a_3, \ldots, a_h=a_1.a_{h-1}, a_2=a_1.a_h
\end{equation}
We also get the following relation between $a_1$ and $a_2$:
\begin{equation*}
a_1^{h-1}.a_2=a_2
\end{equation*}
The equation $a_1^{h-1}.x=x$ holds for the elements of a subgroup which includes the elements of the cycle \eqref{eq61}.

Furthermore, the following relation exists between any two elements of this subgroup,
\begin{equation*}
a^{h-1}.b=b,\text{ so in addition, }a_2^{h-1}.a_1=a_1\text{ and so on.}
\end{equation*}
\begin{defn}
We call $h-1$ the \textit{degree} of the $l$-cycle $a_1,a_2$. Let $A=\left\{a_1,a_2,\ldots,a_v\right\}$ be a simple group, then the degree $G$ of the $l$-cycle of any two elements $a_p,a_q$, $p\neq q$ is a characteristic invariant of this simple group, i.e. it is independent of the choice of $a_p$ and $a_q$.
\end{defn}
\begin{proof}
Any two elements $a_p,a_q\in A$ have a certain $l$-degree $g_{pq}$; since $g_{pq}$ is a natural number, there must a smallest one. Let $a_i,a_k$ be a combination whose cycle-degree is $g$, then $a_i^g.a_k=a_k$. The elements $x$ that solve the equation $a_i^g.x=x$, constitute a subgroup of $A$ which, since it includes both $a_i$ and $a_k$, must necessarily be equal to $A$, since $A$ is simple. Thus,
\begin{equation*}
a_p^g.a_q=a_q\text{ for all elements of } A
\end{equation*}
i.e. any two elements $a_p,a_q\in A$ have the cycle-degree $g$.
\end{proof}
Let $A$ be a simple group of order $N$ and $l$-cycle degree $g$. Then, using $a_1$ and $a_2$, we generate the cycle
\begin{equation}\label{eq62}
a_1,a_2,a_3=a_1.a_2,\ldots,a_{g+1}=a_1.a_g (a_2=a_1.a_{g+1})
\end{equation}
If not all elements in $A$ have been included in this cycle, let $a_{g+2}$ be such an element, not included in \eqref{eq62}. We generate
\begin{equation}\label{eq63}
a_1,a_{g+2}, a_{g+3}=a_1.a_{g+2},\ldots,a_{2g+1}=a_1.a_{2g} (a_{g+2}=a_1.a_{2g+1})
\end{equation}
If $a_i=a_{g+k}, i,k=2,3,\ldots,g+1$, then we would have that
\begin{equation*}
a_1.a_i=a_i+1=a_1.a_{g+k})a_{g+k+1}
\end{equation*}
and finally $a_{g+2}=a_l$, a contradiction. Thus, apart from $a_1$, \eqref{eq62} and \eqref{eq63} are disjoint.

If there are further elements of $A$ which are neither in \eqref{eq62} nor \eqref{eq63}, then we generate another cycle with $a_1,a_{2g+2}$ until every element in $A$ is in one cycle. Every cycle \eqref{eq62}, \eqref{eq63}, $\ldots$ includes $g+1$ elements and thus $N=\sigma g+1$, with $\sigma\in\mathbb{N}$ or, equivalently,
\begin{equation*}
N\equiv1\mod g
\end{equation*}
\begin{thm}
Let $N$ be the order and $g$ be the cycle degree of a simple group, then 
\begin{equation*}
N\equiv1\mod g
\end{equation*}
\end{thm}
The same holds for the $r$-degree, the degree of the \textit{right} cycle; we only have to consider the relation $a_i\circ a_k=a_k.a_i$ instead of $a_i.a_k$ for which the system $A=\left\{a_1,a_2,\ldots,a_v\right\}$ is a distributive group as well.

In general, the $l$ degree is not equal to the $r$ degree.

\begin{proof}
Let $A=\left\{a_0,a_1,\ldots,a_v\right\}$ be a system whose elements create only one cycle,
\begin{equation*}
a_0,a_1,a_2=a_0.a_1,a_3=a_0.a_2,\ldots,a_v=a_0.a_{v-1} (a_0.a_v=a_1)
\end{equation*}
We will show that right-sided distributivity follows from left-sided distributivity and Axioms \ref{Axiom1} and \ref{Axiom2}. Since homogeneity holds for such systems, it is sufficient to show that 
\begin{equation*}
(a_i.a_k).a_0=(a_i.a_0).(a_i.a_0)
\end{equation*}
We denote with $a_\varrho$ the element for which
\begin{equation*}
a_0.a_1=a_\varrho.a_0
\end{equation*}
holds. Via left-sided composition with $a_0$ we get
\begin{equation}\label{eq64}
\begin{cases}a_0.a_{\left[t\right]}=a_{\left[\varrho+t-1\right]}.a_0,& a_{\varrho+t-1}=a_k\text{ fixed}, t=k+1-\varrho\\
a_0.a_{\left[k+1-\varrho\right]}=a_{\left[k\right]}
\end{cases}
\end{equation}
where $\left[i\right]$ is the element for which 
\begin{equation*}
\left[i\right]\equiv i\mod v
\end{equation*}
holds. Now, it follows from \eqref{eq64} and \eqref{eq63} that
\begin{eqnarray*}
(a_t.a_0).(a_s.a_0)&=\left[a_0.a_{\left[t+1-\varrho\right]}\right].\left[a_0.a_{\left[s+1-\varrho\right]}\right]&=a_0.\left[a_{\left[t+1-\varrho\right]}-a_{\left[s+1-\varrho\right]}\right]\\
&=a_0.a_\mu&=a_{\left[\varrho+\mu-1\right]}.a_0
\end{eqnarray*}
where 
\begin{equation*}
a_\mu=a_{\left[t+1-\varrho\right]}.a_{\left[s+1-\varrho\right]}
\end{equation*}
Composed left-sided $(\varrho-1)$ times with $a_0$, this give
\begin{equation*}
a_t.a_s=a_{\left[\mu+\varrho-1\right]}
\end{equation*}
and so \eqref{eq63} becomes
\begin{equation*}
(a_t.a_0).(a_s.a_0)=(a_t.a_s).a_0
\end{equation*}
\end{proof}

\appendix
\section{Translated definitions}
In this section we will present the three definitions that are not used in the form of this paper anymore. We first give the German original definitions and then their respective translations.
All the formatting and spelling has been copied exactly as in the originals.

The translations of the, now obsolete, "names" has been chosen to be as close as possible to their meaning in German. It is, however, not a literal translation. Should a reader recall the correct translation, we would be grateful if they could send a short note to the translator.
\subsection{Isomorphismen}\label{chap:isomorph}
This part is taken from \cite{zbMATH03005477}, Chapter 1 part 11, \textit{Isomorphismen} $\left[\text{isomorphisms}\right]$. 
\begin{defn}
\textit{Sind $\mathfrak{F}$ und $\mathfrak{F}'$ zwei Gruppen, und ist jedem Element $F$ aus $\mathfrak{F}$ ein bestimmtes Element $F'=I(F)$ aus $\mathfrak{F}'$ so zugeordnet, dass stets
\begin{equation*}
I(F_1)I(F_2) = I(F_1F_2)
\end{equation*}
ist und durchl\"auft dabei $F'$ alle Elemente von $\mathfrak{F}'$ wenn $F$ alle Elemente von $\mathfrak{F}$ durchl\"auft, so heißt die Gruppe $\mathfrak{F}'$ isomorph zu $\mathfrak{F}$ und die Abbildung selbst ein Isomorphismus}. Entspricht hierbei zu jedem Element $F'$ auch nur ein einziges Element $F$, so heisst $\mathfrak{F}'$ zu $\mathfrak{F}$ \textit{einstufig isomorph}, anderenfalls \textit{mehrstufig isomorph}.
\end{defn}
Translated, this becomes
\begin{defn}\label{defn:isomorph}
Let $\mathfrak{F}$ and $\mathfrak{F}*$ be two groups and let there be a mapping from each element $F$ in $\mathfrak{F}$ to each element $F=I(F)'$ in $\mathfrak{F}'$ for which
\begin{equation*} 
I(F_1)I(F_2) = I(F_1F_2)
\end{equation*}
holds, and if $F'$ runs through all elements of $\mathcal{F}'$ if $F$ runs through all elements of $\mathfrak{F}$, then the group $\mathfrak{F}'$ is isomorphic to $\mathfrak{F}$ and the map $I$ is an isomorphism. If for each element $F'$ there is exactly one element $F$, then $\mathfrak{F}'$ and $\mathfrak{F}$ are \textit{uniquely isomorphic}, otherwise \textit{v-step isomorphic}.
\end{defn}
\subsection{$R$-gliedrige Gruppe}
This part is taken from \cite{Lie1874}.
\begin{defn}
Der Begriff einer \textit{Gruppe von Transformationen}, welcher zun\"achst in der Zahlentheorie und in der Substitutionstheorie seine Ausbildung fand, ist in neuerer Zeit verschiedentlich auch f\"ur geometrische, resp. allgemeine analytische Untersuchungen verwendet worden. Man sagt von einer Schaar von Transformationen
\begin{equation*}
x_i'=f_i(x_1,\ldots,x_n,a_1,\ldots,a_r)
\end{equation*}
(wobei die $x$ die urspr\"unglichen, die $x'$ die neuen Voriablen (sic!) und die $a$ Parameter bedeuten, die im folgenden stets \textit{continuirlich} ver\"anderlich gedacht werden), dass sie eine \textit{r-gliedrige Gruppe} bilden, wenn irgend zwei Transformationen der Schaar zusammengesetzt wieder eine der Schaar angeh\"orige Transformation ergeben, wenn also aus den Gleichungen
\begin{equation*}
x_i'=f_i(x_1\cdots x_n\;\alpha_1\cdots \alpha_r)
\end{equation*}
und
\begin{equation*}
x_i''=f_i(x'_1 \cdots x'_n\;\beta_1\cdots\beta_r)
\end{equation*}
hervorgeht:
\begin{equation*}
x_i''=f_i(x_1\cdots x_n\;\gamma_1\cdots\gamma_r)
\end{equation*}
unter den $\gamma$ Gr\"ossen verstanden, die nur von den $\alpha,\beta$ abh\"angen.
\end{defn}
Translated, this becomes
\begin{defn}\label{defn:rgliedrig}

$\left[\ldots\right]$ A set of transformations
\begin{equation*}
x_i'=f_i(x_1,\ldots,x_n,\alpha_1,\ldots,\alpha_r)
\end{equation*}
(where the $x$ are the old, $x'$ are the new variables and the $\alpha$ are parameters which are thought to be continuous) constitutes an \textit{r-parameter group} if the composition of any two transformations in this set is also in this set, that is if from the equations
\begin{equation*}
x_i'=f_i(x_1,\ldots,x_n,alpha_1,\ldots,alpha_r)
\end{equation*}
and
\begin{equation*}
x_i''=f_i(x'_1,\ldots,x'_n,\beta_1,\ldots,\beta_r)
\end{equation*}
it follows that
\begin{equation*}
x_i''=f_i(x_1,\ldots,x_n,\gamma_1,\ldots,\gamma_r)
\end{equation*}
with parameters $\gamma$ that only depend on $\alpha,\beta$.
\end{defn}

\end{document}